\documentclass[12pt]{article}
\usepackage{amsmath,amssymb,amsthm,amsfonts,bm,latexsym,graphicx,float, caption, subcaption,yfonts,cite}
\usepackage{setspace}
\def\R{\Bbb{R}}
\usepackage[font={small}]{caption}

\newtheorem{theorem}{Theorem}
\newtheorem{prop}[theorem]{Proposition}
\newtheorem{lemma}[theorem]{Lemma}
\newtheorem{remark}[theorem]{Remark}
\newtheorem{corollary}[theorem]{Corollary}
\newtheorem{definition}[theorem]{Definition}

\everymath{\displaystyle}

\title{Some Inversion Formulas for the Cone Transform}
\author{Fatma Terzioglu\thanks{Department of Mathematics, Texas A$\&$M University, College Station, TX 77843-3368, USA, e-mail: fatma@math.tamu.edu}}
\date{\vspace{5ex}}

\begin{document}
\maketitle
\begin{abstract}
Several novel imaging applications have lead recently to a variety of Radon type transforms, where integration is done over a family of conical surfaces. We call them \emph{cone transforms} (in 2D they are also called \emph{V-line} or \emph{broken ray} transforms). Most prominently, they are present in the so called Compton camera imaging that arises in medical diagnostics, astronomy, and lately in homeland security applications. Several specific incarnations of the cone transform have been considered separately. In this paper, we address the most general (and overdetermined) cone transform, obtain integral relations between cone and Radon transforms in $\mathbb{R}^n$, and a variety of inversion formulas. In many applications (e.g., in homeland security), the signal to noise ratio is very low. So, if overdetermined data is collected (as in the case of Compton imaging), attempts to reduce the dimensionality might lead to essential elimination of the signal. Thus, our main concentration is on obtaining formulas involving overdetermined data.
\end{abstract}

\section{Introduction}
In this paper, we study the so called \emph{cone transform}, where a function on $\R^n$ is integrated over various conical surfaces (in 2D, the names \emph{V-line transform} and \emph{broken ray transform} are also used). Such transforms arise in a variety of new imaging techniques, e.g. in optical imaging \cite{Florescu}, but most prominently in the so called \emph{Compton camera imaging}, which we will briefly explain now.
The conventional gamma cameras used in medical SPECT(Single Photon Emission Tomography) imaging determine the direction of an incoming $\gamma$-photon by "collimating'' the detector (see Fig. \ref{fig:collimation&scatter}(left)). This considerably decreases the efficiency, because only a small portion of the incoming $\gamma$-rays passes through the collimator \cite{Basko}. Thus, the acquired signal is weak and statistically noisy. The situation is similar in astronomy and even more severe in homeland security applications \cite{KuchCBMS,ADHKK,ACCHKOR,Xun}.

On the other hand, Compton cameras utilize Compton scattering (see Fig. \ref{fig:collimation&scatter}(right)) and use electronic rather than mechanical collimation to provide simultaneous multiple views of the object and dramatic increase in sensitivity \cite{Singh}.
\begin{figure}[H]
        \centering
        \begin{subfigure}[b]{0.4\textwidth}
                \includegraphics[width=\textwidth]{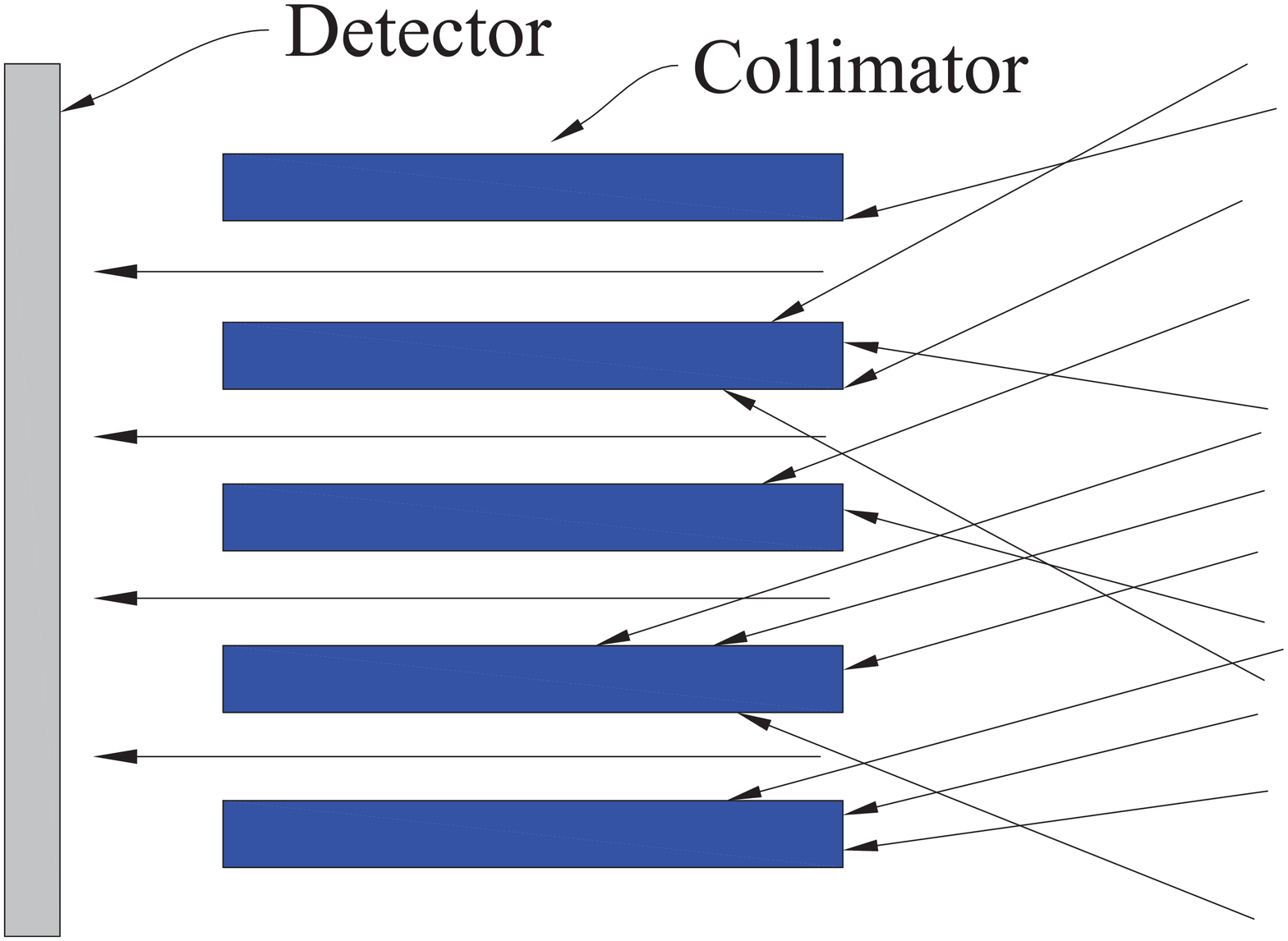}
        \end{subfigure}
        \begin{subfigure}[b]{0.4\textwidth}
                \includegraphics[width=\textwidth]{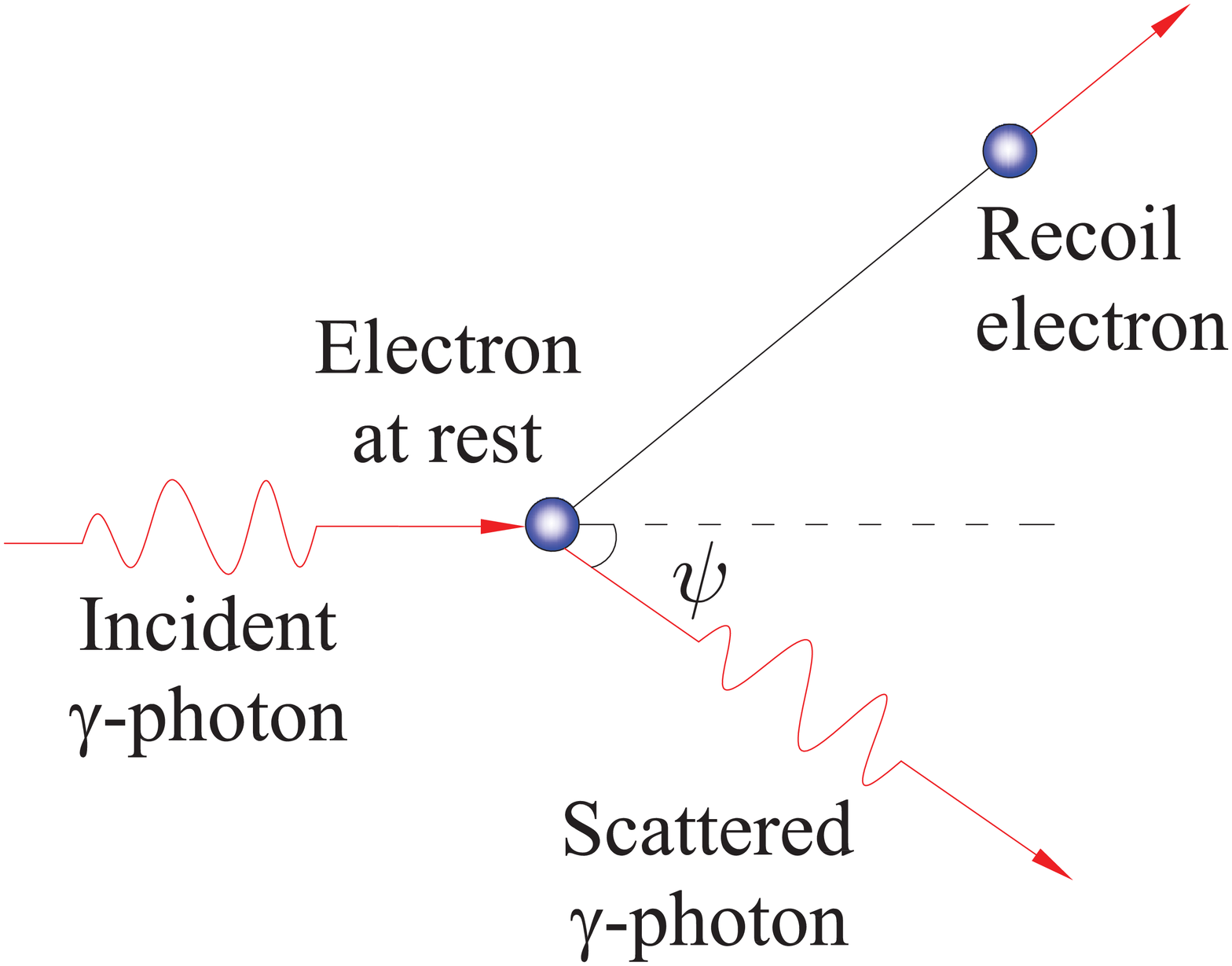}
        \end{subfigure}
        \caption{Left: Collimation. Right: Compton Scattering.}\label{fig:collimation&scatter}
\end{figure}

A Compton camera consists of two parallel detectors (see Fig. \ref{fig:compton camera}). When the photon hits the first detector, where its position $u$ and energy $E_1$ are recorded, it undergoes Compton scattering. Then, it is absorbed in the second detector where its position $v$ and energy $E_2$ are again measured. The scattering angle $\psi$ and a unit vector $\beta$ are calculated from the data as follows (see e.g. \cite{Todd}):
\begin{equation}
\cos\psi=1-\frac{mc^2E_1}{(E_1+E_2)E_2} \quad \quad \quad \quad \beta=\frac{u-v}{|u-v|}.
\end{equation}
Here, $m$ is the mass of the electron and $c$ is the speed of light.

From the knowledge of the scattering angle $\psi$ and the vector $\beta$, we conclude that the photon originated from the surface of the cone with central axis $\beta$, vertex $u$ and opening angle $\psi$ (see Fig. \ref{fig:compton camera}). Therefore, although the exact incoming direction of the detected particle is not available, one knows a surface cone of such possible directions.  One can argue that the data provided by Compton camera are integrals of the distribution of the radiation sources over conical surfaces having vertex at the detector. The operator that maps source intensity distribution function $f(x)$ to its integrals over these cones is called the \emph{cone} or \emph{Compton transform}. The goal of Compton camera imaging is to recover source distribution from this data \cite{ADHKK}.
\begin{figure}[H]
\begin{center}
\includegraphics[width=3.3in,height=2.5in]{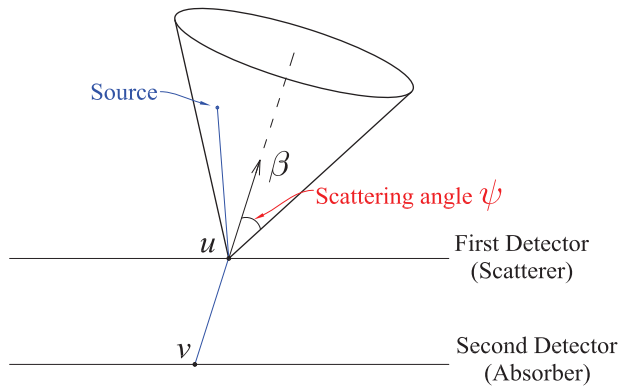}
\caption{Schematic representation of a Compton camera.}
\label{fig:compton camera}
\end{center}
\end{figure}
In the Compton camera imaging applications mentioned above, the vertex of the cone is located on the detector plane, while in other applications vertices are not restricted, although some other conditions are imposed on the cones. We thus find it useful to understand analytic properties of a more general cone transform, where no restriction on the vertex location is imposed. This is the transform addressed in this text with the hope that it can be useful for more restricted versions. As for instance Remark \ref{R:comptonappl} shows, one indeed arrives at applications to the Compton imaging\footnote{It is planned to address these applications in detail elsewhere.}.

The problem of inverting the cone transform is over-determined. For instance, the space of 2D cones with vertices on a linear detector array is three-dimensional, and the space of 3D cones with vertices on a detector surface is five-dimensional. Without the restriction on the vertex, the dimensions are correspondingly four and six. One thus is tempted to restrict the set of cones, in order to get a non-over-determined problem. There exist several inversion formulas of this type (e.g. \cite{Basko, Cree, Moon, NgTr2005}). However, as we have already mentioned, when the signals are weak (e.g. in homeland security applications (e.g., \cite{ADHKK}), restricting the data would lead to essential elimination of the signal. We thus intend to use the full data set.

Probably, the first known analytical reconstruction formula in 3D was given in \cite{Cree}, where the authors considered cones with vertical axis only. The papers \cite{Basko, JungMoon} contain spherical harmonics expansion solutions. Another inversion formula for cone transforms on cones having fixed central axis and variable opening angle is provided in \cite{NgTr2005}. The paper \cite{Smith} presents two reconstruction methods for two Compton data models. The complete set of data was used in \cite{Maxim, Maxim2014}. Inversion formulas for $n$-dimensional cone transform over vertical cones are provided in \cite{ Gouia-Zarrad, Haltmeier}. All these works only addressed the cones with the vertex on the detector. Inversion algorithms for various 2D cone transforms are given in \cite{Basko, Florescu, Gouia-Zarrad-Ambarts, Hristova, Morvidone}.

In this paper, we derive various inversion formulas\footnote{The reader should recall that it is common to have a variety of different inversion formulas for Radon type transforms, which are all the same for perfect data, but react differently to unavoidable errors in data \cite{Natt_old,KuchCBMS}. Having such a variety is even more important when dealing with overdetermined data, as in Compton imaging.} for the full data cone transform in $\mathbb{R}^n$. In Section 2, we define the cone transform and state its basic properties. In Section 3, we obtain an integral relation between the cone and Radon transforms in $\mathbb{R}^n$ and deduce from it an inversion formula for the cone transform. In Section 4, we provide a different inversion formula derived from another integral relation between the cone and Radon transforms in $\mathbb{R}^n$. Both of these formulas provide reconstructions only at vertices of the cones which is an inconvenience for Compton imaging. However, the integral relation provided in Section 4 also enables us to associate the cone transform with the cosine transform. This result is given in Section 5, and through this relation, we obtain the Radon transform explicitly in terms of the cone transform in Theorem 14 which leads to a variety of inversion algorithms from Compton data as discussed in Remark 15. The results of a numerical simulation for $n=2$ are also provided. In Section 6, we investigate the relationship between the cone transform and spherical harmonics. Finally, we prove some auxiliary technical results in Section 7.

\section{Definition and Basic Properties of the Cone Transform}
A round cone in $\mathbb{R}^n$ can be parametrized by a tuple $(u, \beta, \psi)$, where $u \in \mathbb{R}^n$ is the cone vertex, vector $\beta \in S^{n-1}$ is directed along the cone's central axis, and $\psi \in (0,\pi)$ is the opening angle of the cone (see Fig. \ref{fig:compton camera}). Then, a point $x \in \mathbb{R}^n$ lies on the cone iff
\begin{equation}\label{cone eqn}
 (x-u)\cdot\beta=|x-u|\cos \psi.
\end{equation}

The \emph{$n$-dimensional cone transform}  $C$ maps a function $f$ into the set of its integrals over the circular cones in $\mathbb{R}^n.$ Explicitly,
\begin{equation}
\label{cone trans}
 Cf(u,\beta,\psi)=\int\limits_{(x-u)\cdot\beta=|x-u|\cos \psi}f(x)dx
\end{equation}
where $dx$ is the surface measure on the cone.

The \emph{$n$-dimensional vertical cone transform} maps a function $f$  into the set of its integrals over the cones having central axis parallel to the $x_n$-axis, and thus the vector $\beta$ is equal to $e_n=(0,...,0,1) \in \mathbb{R}^n$. It can be written in terms of the spherical coordinates. Namely,
\begin{equation}
\label{nD_vertical cone}
Cf(u, e_n, \psi)=\int\limits_0^\infty \int\limits_{S^{n-2}} f(u+ \rho((\sin\psi) \omega,\cos \psi))(\rho \sin\psi)^{n-2}d\omega d\rho.
\end{equation}

In two dimensions, the equation \eqref{cone eqn} describes two rays with a common vertex (see Fig. \ref{fig:2Dcone}). A cone in two dimensions can be parametrized by a point $u \in \mathbb{R}^2$ that serves as its vertex, an opening angle $\psi \in (0, \pi)$ and a vector $\beta=\beta(\phi)=(\sin\phi, \cos \phi) \in S^1$ directed along the central axis.
\begin{figure}[H]
\begin{center}
\includegraphics[width=2.6in,height=1.8in]{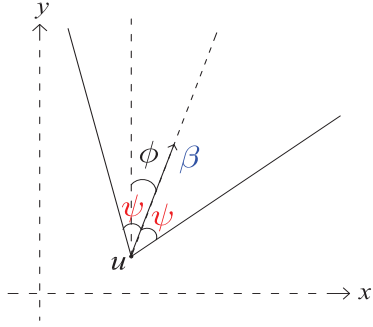}
\caption{A Cone in 2-dimensions.}
\label{fig:2Dcone}
\end{center}
\end{figure}

Then, the 2D cone transform of a function $f \in \mathcal{S}(\mathbb{R}^2)$ is given by
\begin{align}
\label{2D_cone}
\begin{split}
   Cf(u, \beta, \psi)= Cf(u, \beta(\phi), \psi)&=\int\limits_0^\infty f(u+r(\sin(\psi+\phi),\cos(\psi+\phi)))dr  \\
    &+ \int\limits_0^\infty f(u+r(-\sin(\psi-\phi),\cos(\psi-\phi)))dr.
\end{split}
\end{align}

As a straightforward calculation shows, analogously to the Radon transform, cone transform has an evenness property, and is shift and rotation invariant:
\begin{lemma} \label{Properties} Let $f \in \mathcal{S}(\mathbb{R}^n)$, $u \in \mathbb{R}^n$, $\beta \in S^{n-1}$ and $\psi \in (0,\pi)$. Then,
 \begin{enumerate}
 \item
  \begin{equation}\label{evenness}
  Cf(u,-\beta,\psi)=Cf(u,\beta,\pi-\psi).
  \end{equation}
  \item Let $T_a$ be the translation operator in $\mathbb{R}^n$, defined as $T_af(x)=f(x+a)$ for $a \in \mathbb{R}^n$. We define $$T_a(Cf)(u,\beta,\psi):=Cf(u+a,\beta,\psi).$$
  Then,
  $$T_aC=CT_a.$$
  \item Let $A$ be an $n\times n$ rotation matrix and $M_Af(x)=f(Ax)$ be the corresponding rotation operator. We define
  $$M_A(Cf)(u,\beta,\psi):=Cf(Au,A\beta,\psi).$$
  Then,
  $$M_AC=CM_A.$$
 \end{enumerate}
\end{lemma}

\section{Inversion of the Cone Transform}
 In the following, we investigate the relation between the cone and Radon transforms and provide various analytical inversion formulas for the $n$-dimensional cone transform.

 We first recall that the n-dimensional Radon transform $R$ maps a function $f$ on $\mathbb{R}^n$ into the set of its integrals over the hyperplanes of $\mathbb{R}^n$. Namely, if $\omega \in S^{n-1}$ and $s \in \mathbb{R}$,
 \begin{equation}
\label{Def of Radon}
Rf(\omega, s)= \int\limits_{x \cdot \omega=s}f(x)dx.
\end{equation}
In this setting, the Radon transform of $f$ is the integral of $f$ over the hyperplane orthogonal to $\omega$ with signed distance $s$ from the origin.

The Radon transform is invertible on  $\mathcal{S}(\mathbb{R}^n)$, namely
\begin{equation}\label{inverse_radon}
 f=\frac{1}{2}(2\pi)^{1-n}I^{-\alpha}R^\#I^{\alpha-n+1}Rf, \quad \quad \alpha < n.
\end{equation}
Here, $R^\#$ is the back projection operator, and $I^\alpha$, $\alpha<n$, is the \emph{Riesz potential} acting on a function $f(u)$ as
$$\widehat{(I^\alpha f)}(\xi)=|\xi|^{-\alpha}\hat{f}(\xi),$$
where $\hat{f}$ is the Fourier transform of $f$. For instance, when $n$ is odd, $I^{1-n}$ is simply the differential operator
$$I^{1-n}=(-\Delta)^{(n-1)/2}$$
with $\Delta$ being the Laplacian (see e.g. \cite{Natt_old}).

\begin{theorem}\label{Inversion Theorem1}
Let $f \in \mathcal{S}(\mathbb{R}^n)$. Then,
\begin{enumerate}
 \item For any $u \in \mathbb{R}^n$ and $\beta \in S^{n-1}$, we have
 \begin{equation}\label{int_rel1}
  \int\limits_0^\pi Cf(u,\beta,\psi)d\psi=\frac{\Gamma(\frac{n-1}{2})}{2\pi^{(n-1)/2}}\int\limits_{S^{n-1}}Rf(\omega, u\cdot \omega)d\omega=\frac{\Gamma(\frac{n-1}{2})}{2\pi^{(n-1)/2}}R^\# Rf(u).
  \end{equation}
 \item Let a function $\mu: S^{n-1} \to \mathbb{R}$ be such that $\int\limits_{S^{n-1}} \mu(\beta)d\beta=1$. For any $f \in \mathcal{S}(\mathbb{R}^n)$,
 \begin{equation}\label{inversion formula}
f(u)=\frac{\pi^{-n/2}\Gamma(\frac{n}{2})}{2\Gamma(n-1)}\int\limits_{S^{n-1}}\int\limits_0^\pi I^{1-n}Cf(u, \beta, \psi)\mu(\beta) d\psi d\beta.
\end{equation}
\end{enumerate}
\end{theorem}

\begin{remark}\label{R:beta}\indent
\begin{enumerate}
\item One notices that according to (\ref{int_rel1}), the inversion formula (\ref{inversion formula}) consists of a backprojecting of the cone data, followed by a filtration (i.e., is what is called a FBP type formula).
\item One can choose $\mu(\beta)$ to be equal to a delta-function, which would eliminate integration with respect to $\beta$ in (\ref{inversion formula}). However, if the signal is very week, eliminating almost all values of $\beta$ would lead to elimination of the signal. Thus weighted integration with respect to $\beta$ allows for accounting for all data collected.
\end{enumerate}
\end{remark}
\begin{proof} We first prove the theorem for dimensions $n \geq 3$.
 \begin{align*}
 \int\limits_0^\pi &Cf(u, e_n, \psi)d\psi=
 \int\limits_0^\pi \int\limits_0^\infty \int\limits_{S^{n-2}} f(u+\rho((\sin\psi) \omega,\cos \psi))(\rho \sin\psi)^{n-2}d\omega d\rho d\psi\\
 &=\int\limits_{S^{n-1}}\int\limits_0^\infty  f(u+\rho\sigma)\rho^{n-2}d\rho d\sigma
 =\int\limits_{\mathbb{R}^n} f(u+x)|x|^{-1}dx = \frac{1}{|S^{n-2}|} R^\#Rf(u),
 \end{align*}
 The last equality is due to \cite[Chapter 2, Theorem 1.5]{Natt_old} (see also Corollary \ref{BPR}).
 As both $R$ and $R^\#$ commute with rigid motions in $\mathbb{R}^n$, we obtain for any $\beta \in S^{n-1}$,
 $$\int\limits_0^\pi Cf(u, \beta, \psi)d\psi =\frac{1}{|S^{n-2}|} R^\#Rf(u).$$
Thus, for any function $\mu$ on $S^{n-1}$ such that $\int\limits_{S^{n-1}}\mu(\beta)d\beta=1$, we have
 $$\int\limits_{S^{n-1}}\int\limits_0^\pi Cf(u, \beta, \psi)\mu(\beta) d\psi d\beta =\frac{1}{|S^{n-2}|} R^\#Rf(u)= \frac{\Gamma(\frac{n-1}{2})}{2\pi^{(n-1)/2}}R^\#Rf(u).$$
Note that the last equality follows from the area formula for the $n$-sphere, that is
 \begin{equation}\label{area of nsphere}
  |S^{n-1}|=\frac{2\pi^{n/2}}{\Gamma(\frac{n}{2})}.
 \end{equation}
 Using \eqref{inverse_radon} with $\alpha=n-1$, and utilizing the duplication formula (see e.g. \cite{Szego})
 \begin{equation}\label{duplication fmla}
\Gamma(z)\Gamma(z+\frac{1}{2})=2^{1-2z}\sqrt{\pi}\Gamma(2z),
 \end{equation}
 we conclude that
 \begin{align*}
 f(u)=\frac{\pi^{n/2}\Gamma(\frac{n}{2})}{2\Gamma(n-1)}\int\limits_{S^{n-1}}\int\limits_0^\pi I^{1-n}Cf(u, \beta, \psi)\mu(\beta) d\psi d\beta.\\
 \end{align*}

For the 2-dimensional case, we only need to provide the proof of \eqref{int_rel1}, since the rest of the proof stays the same. Assume for now that $u=0$. By definition of the 2D cone transform, we have
 \begin{align*}
 \int\limits_0^\pi Cf(0,\beta(\phi),\psi)d\psi&=\int\limits_0^\pi \int\limits_0^\infty f(r\sin(\psi+\phi),r\cos(\psi+\phi))dr d\psi \\
    &+ \int\limits_0^\pi \int\limits_0^\infty f(-r\sin(\psi-\phi),r\cos(\psi-\phi))dr d\psi.
 \end{align*}
Changing variables, we obtain
\begin{equation*}
 \int\limits_0^\pi f(r\sin(\psi+\phi),r\cos(\psi+\phi)) d\psi=\int\limits_\phi^{\pi+\phi} f(r\sin\psi,r\cos\psi) d\psi,
\end{equation*}
and
\begin{equation*}
 \int\limits_0^\pi f(-r\sin(\psi-\phi),r\cos(\psi-\phi)) d\psi=\int\limits_{-\pi+\phi}^\phi f(r\sin\psi,r\cos\psi) d\psi.
\end{equation*}
Thus,
\begin{align*}
 \int\limits_0^\pi Cf(0,\beta(\phi),\psi)d\psi=\int\limits_0^\infty \int\limits_{-\pi+\phi}^{\pi+\phi} f(r\sin\psi,r\cos\psi) d\psi dr.
\end{align*}
Changing variables by letting $\theta=\frac{\pi}{2}-\psi$ and using $2\pi$-periodicity of sine and cosine functions, we get
\begin{equation*}
\int\limits_{-\pi+\phi}^{\pi+\phi} f(r\sin\psi,r\cos\psi) d\psi
=\int\limits_{-\frac{\pi}{2}-\phi}^{\frac{3\pi}{2}-\phi} f(r\cos\theta,r\sin\theta) d\theta =\int\limits_0^{2\pi} f(r\cos\theta,r\sin\theta) d\theta.
\end{equation*}
Therefore,
$$ \int\limits_0^\pi Cf(0,\beta(\phi),\psi)d\psi= \int\limits_{0}^{2\pi}\int\limits_0^\infty f(r\cos\theta,r\sin\theta)drd\theta
=\frac{1}{2}\int\limits_0^{2\pi}Rf(\theta,0) d\theta,$$
where the last equality follows by letting $n=2$ and $p=0$ in \eqref{Asgeirsson}.
Now, using the shift invariance of both cone and Radon transforms, we conclude that
\begin{equation*}\label{integral relation_2D}
 \int\limits_0^\pi Cf(u,\beta,\psi)d\psi=\frac{1}{2}\int\limits_{S^1}Rf(\omega,u\cdot\omega) d\omega=\frac{1}{2}R^\# Rf(u),
\end{equation*}
which is \eqref{int_rel1} with $n=2$, so we are done.
\end{proof}

 \begin{corollary}For $n=3$, the formula \eqref{inversion formula} becomes
$$f(u)=-\frac{1}{4\pi}\int\limits_{S^{n-1}} \int\limits_0^\pi \Delta Cf(u,\beta,\psi)\mu(\beta) d\psi d\beta, $$
where $\Delta$ acts on the variable $u$.
\end{corollary}

\section{An Alternative Inversion Formula}
\label{sec:Another Inversion Formula}
For the derivation of an alternative inversion formula, we need the following relation between the cone and Radon transforms.
\begin{theorem}\label{Integral Relation}
 Let $f \in \mathcal{S}(\mathbb{R}^n)$. For any $u \in \mathbb{R}^n$ and $\beta \in S^{n-1}$, we have
\begin{equation}\label{int_rel}
\int\limits_0^\pi Cf(u,\beta,\psi)\sin\psi d\psi= \frac{\pi}{|S^{n-1}|}\int\limits_{S^{n-1}} Rf(\omega,\omega\cdot u)|\omega\cdot \beta |d\omega,
\end{equation}
where $|S^{n-1}|$ denotes the area of the sphere $S^{n-1}$.
\end{theorem}

As in the case of the Radon transform, invariance properties play a key role in the inversion of the cone transform. In fact, due to rotational invariance, it suffices to prove \eqref{int_rel} only for the vertical cone transform. Moreover, shift invariance enables us to consider vertical cones having vertex at the origin only, that is $u=0$.

\begin{prop}\label{int_rel_vertical}
For any $f \in \mathcal{S}(\mathbb{R}^n)$, we have
\begin{equation} \label{int rel vertical}
\int\limits_0^\pi Cf(0,e_n,\psi)\sin\psi d\psi= \frac{\pi}{|S^{n-1}|}\int\limits_{S^{n-1}} Rf(\omega,0)|\omega \cdot e_n|d\omega.
\end{equation}
\end{prop}
For the proof, see Section~\ref{subsec:proof of prop}.\\

\emph{Proof of Theorem \ref{Integral Relation}}. We will use Proposition \ref{int_rel_vertical} and the properties of the cone transform to deduce Theorem \ref{Integral Relation}. We first remind that the Radon transform commutes with shifts and rotations, that is $R(T_uf)(\omega,s)=Rf(\omega,s+\omega\cdot u)$ and $M_ARf(\omega,s)=Rf(A\omega,s)=R(M_Af)(\omega,s)$.

As cone transform also commutes with shifts, Proposition \ref{int_rel_vertical} implies that
 \begin{align*}
  \int\limits_0^\pi &Cf(u,e_n,\psi)\sin\psi d\psi
  = \int\limits_0^\pi C(T_uf)(0,e_n,\psi)\sin\psi d\psi\\
  &= \frac{\pi}{|S^{n-1}|} \int\limits_{S^{n-1}}R(T_uf)(\omega,0)|\omega \cdot e_n|d\omega= \frac{\pi}{|S^{n-1}|}\int\limits_{S^{n-1}} Rf(\omega,\omega\cdot u)|\omega\cdot e_n|d\omega.
 \end{align*}

Next, for $\beta \in S^{n-1}$, let $A$ be the rotation matrix such that $\beta=Ae_n$ and $x=A^{-1}u$. As cone transform commutes with rotations, we further have
\begin{align*}
\int\limits_0^\pi Cf(u,\beta,\psi)\sin\psi d\psi&=  \int\limits_0^\pi C(M_Af)(x,e_n,\psi) \sin\psi d\psi \\
&= \frac{\pi}{|S^{n-1}|}\int\limits_{S^{n-1}} R(M_Af)(\omega,\omega\cdot x)|\omega\cdot e_n|d\omega.
\end{align*}
Due to the rotational invariance of the Radon transform, we have
\begin{align*}
\int\limits_{S^{n-1}} &R(M_Af)(\omega,\omega\cdot x)|\omega\cdot e_n|d\omega
  =  \int\limits_{S^{n-1}} M_ARf(\omega,\omega\cdot x)|\omega\cdot e_n|d\omega\\
  &=  \int\limits_{S^{n-1}} Rf(A\omega,\omega\cdot x)|\omega\cdot e_n|d\omega
  =  \int\limits_{S^{n-1}} Rf(A\omega,\omega\cdot A^{-1}u)|\omega\cdot A^{-1}\beta|d\omega\\
 & =  \int\limits_{S^{n-1}} Rf(A\omega,A\omega\cdot u)|A\omega\cdot \beta|d\omega
  =  \int\limits_{S^{n-1}} Rf(\omega,\omega\cdot u)|\omega\cdot \beta|d\omega,
\end{align*}
The last equality is due to the rotational invariance of the Lebesgue measure on the sphere. Hence, we obtain \eqref{int_rel}. \qed

\begin{remark}
 As it will be mentioned in Section~\ref{sec:Remarks}, the assumption $f \in \mathcal{S}(\mathbb{R}^n)$ can be significantly weakened. The same applies to Theorem \ref{Inversion Thm}.
\end{remark}

The equality \eqref{int_rel} enables us to invert the cone transform by utilizing the inversion formulas for the Radon transform.
\begin{theorem}\label{Inversion Thm}
Let $f \in \mathcal{S}(\mathbb{R}^n)$. For any $u \in \mathbb{R}^n$, we have
\begin{equation}\label{Inversion}
f(u)=\frac{\Gamma^2(\frac{n+1}{2})}{2\pi^n\Gamma(n)}\int\limits_{S^{n-1}} \int\limits_0^\pi I^{1-n} Cf(u,\beta,\psi)\sin\psi d\psi d\beta.
\end{equation}
\end{theorem}
\begin{proof} Integrating both sides of \eqref{int_rel} with respect to $\beta$ over $S^{n-1}$, we obtain
 \begin{align*}
\int\limits_{S^{n-1}} \int\limits_0^\pi Cf(u,\beta,\psi)\sin\psi d\psi d\beta= \frac{\pi}{|S^{n-1}|}\int\limits_{S^{n-1}}
Rf(\omega,\omega\cdot u)\int\limits_{S^{n-1}}|\omega \cdot \beta|d\beta d\omega.
 \end{align*}
Using the rotation invariance of the Lebesgue measure on the sphere, for any $\omega \in S^{n-1}$,  we compute
 \begin{align*}
 \int\limits_{S^{n-1}}|\omega \cdot \beta|d\beta&=\int\limits_{S^{n-2}}\int\limits_0^\pi |\cos \phi|(\sin \phi)^{n-2}d\phi d\theta =\frac{2|S^{n-2}|}{n-1}.
 \end{align*}
 Thus, we get
  \begin{align}\label{BPR_rel}
  \int\limits_{S^{n-1}} \int\limits_0^\pi &Cf(u,\beta,\psi)\sin\psi d\psi d\beta \nonumber
  =\frac{\pi}{|S^{n-1}|} \frac{2|S^{n-2}|}{n-1} \int\limits_{S^{n-1}}Rf(\omega,\omega\cdot u)d\omega \\
  &=\frac{2\pi}{n-1} \frac{|S^{n-2}|}{|S^{n-1}|} \int\limits_{S^{n-1}} R^{\#} Rf(u)
 =\frac{\pi\Gamma(n)}{2^{n-1}\Gamma^2(\frac{n+1}{2})}R^{\#} Rf(u).
 \end{align}
Note that, for the evaluation of the constant, we have used the area formula for the $n$-sphere, \eqref{area of nsphere} and the duplication formula \eqref{duplication fmla}.
 Now, using formula \eqref{inverse_radon} with $\alpha=n-1$, we obtain the result.
 \end{proof}
 \begin{corollary}For $n=3$, the formula \eqref{Inversion} reads as
$$f(u)=\frac{-1}{4\pi^3}\int\limits_{S^{n-1}} \int\limits_0^\pi \Delta Cf(u,\beta,\psi)\sin\psi d\psi d\beta, $$
where $\Delta$ acts on the variable $u$.
\end{corollary}

\section{Relation of the Cone Transform with Cosine Transform. Other Inversion Formulas}
The main goal of this section is to derive a formula which is applicable in Compton imaging. This is achieved in Theorem 14 and Remark 15. We start, however, with a relation between the cone transform and the cosine transform which is defined as follows:
\begin{definition}
The cosine transform of a function $f \in C(S^{n-1})$ is defined by
\begin{equation}
 \textswab{C}f(\omega)=\frac{1}{|S^{n-1}|}\int\limits_{S^{n-1}}f(\sigma)|\sigma \cdot \omega|d\sigma,
\end{equation}
for all $\omega \in S^{n-1}$.
\end{definition}

Now the relation \eqref{int_rel} can be written as
\begin{align} \label{cos of radon}
\textswab{C}(R(T_uf))(\beta)
= \frac{1}{|S^{n-1}|}\int\limits_{S^{n-1}} R(T_uf)(\omega,0)|\omega\cdot \beta |d\omega
= \frac{1}{\pi}\int\limits_0^\pi Cf(u,\beta,\psi)\sin\psi d\psi.
\end{align}

The cosine transform is a continuous bijection of $C_{even}^\infty (S^{n-1})$ to itself (see e.g. \cite{Gardner}, \cite{Rubin}).  Since, for any $f \in \mathcal{S}(\mathbb{R}^n)$, $Rf(\omega, 0)$ is an even function in $C^\infty(S^{n-1})$, we can recover the function $R(T_uf)$ by inverting the cosine transform. Before stating this inversion formula, we recall the definitions of the Beltrami-Laplace operator and the Funk transform.

\begin{definition}Let $f \in C^2(S^{n-1})$. The Beltrami-Laplace operator $\Delta_S$ on $S^{n-1}$ is defined by
\begin{equation}\label{Beltrami-Laplace}
 (\Delta_Sf)(\frac{x}{|x|})= |x|^2(\Delta \tilde{f})(x),
\end{equation}
where $\tilde{f}(x)=f(\frac{x}{|x|})$ is the homogeneous extension of $f$ to $\mathbb{R}^n$, and $\Delta$ is the Laplace operator on $\mathbb{R}^n$.
\end{definition}

\begin{definition}
Funk transform of a function $f \in C(S^{n-1})$ is defined by
\begin{equation}\label{def_Funk}
 Ff(\theta)=\int\limits_{S^{n-1}\cap \theta^\bot} f(\sigma)d_\theta \sigma=\int\limits_{\{\sigma \in S^{n-1}: d(\sigma,\theta)=\pi/2\}} f(\sigma)d_\theta \sigma.
\end{equation}
Here, $d(\sigma,\theta)=arccos(\sigma \cdot \theta)$ is the geodesic distance between the points $\sigma$ and $\theta$ in $S^{n-1}$, and $d_\theta \sigma$ stands for the $O(n)$-invariant probability measure on the $(n-2)$-dimensional sphere $S^{n-1}\cap \theta^\bot$.
\end{definition}

\begin{theorem}\cite{Rubin} \label{cosine inversion}
 Let $g=\textswab{C}f$, $f \in C_{even}^\infty (S^{n-1})$. Then, if $n$ is odd,
 \begin{equation}\label{Inversion_cosine_odd dim}
  f(\omega)=P_r(\Delta_S)\left\{\frac{-2\pi^{(2-n)/2}}{\Gamma(\frac{n}{2})}\int\limits_{S^{n-1}}g(\sigma)\log{\frac{1}{|\omega\cdot \sigma|}}d\sigma \right\}
  +\frac{\Gamma(\frac{n+1}{2})}{\pi^{(n-1)/2}}\int\limits_{S^{n-1}}g(\sigma)d\sigma,
 \end{equation}
 with $r=(n+1)/2$, and if $n$ is even,
\begin{equation}\label{Inversion_cosine_even dim}
 f=cP_r(\Delta_S)Fg, \quad \quad c=-\frac{\pi 2^{n-1}}{\Gamma(n-1)},
\end{equation}
with $r=n/2$, where $F$ is the Funk transform and
$$P_r(\Delta_S)=4^{-r}\prod_{k=0}^{r-1}\left[-\Delta_S+(2k-1)(n-1-2k)\right],$$
with $\Delta_S$ being the Beltrami-Laplace operator on $S^{n-1}$.
\end{theorem}

Thus, we can find $RT_uf$ explicitly for all $u \in \mathbb{R}^n$:
\begin{theorem} \label{Radon in terms of cone}
Let $f \in \mathcal{S}(\mathbb{R}^n)$. For any $u \in \mathbb{R}^n$ and $\omega \in S^{n-1}$,
\begin{enumerate}
 \item if $n$ is odd,
  \begin{align}\label{Radon by cone_odd}
 &Rf(\omega, \omega \cdot u)\nonumber \\
 &=\frac{-2\pi^{-n/2}}{\Gamma(\frac{n}{2})}P_{(n+1)/2}(\Delta_S)\left\{\int\limits_{S^{n-1}} \int\limits_0^\pi Cf(u,\beta,\psi)\log{\frac{1}{|\omega\cdot \beta|}}\sin\psi d\psi d\beta \right\}\nonumber \\
  &+\frac{\Gamma(\frac{n+1}{2})}{\pi^{(n+1)/2}}\int\limits_{S^{n-1}}\int\limits_0^\pi Cf(u,\beta,\psi)\sin\psi d\psi d\beta,
  \end{align}
 \item if $n$ is even,
  \begin{equation}\label{Radon by cone_even}
  Rf(\omega, \omega \cdot u)=\frac{-2^{n-1}}{\Gamma(n-1)}\int\limits_0^\pi P_{n/2}(\Delta_S)F(Cf)(u,\omega,\psi)\sin\psi d\psi,
  \end{equation}
\end{enumerate}
where $F$ and $P_r(\Delta_S)$ are given as in Theorem \ref{cosine inversion}, and both of them act on the variable $\omega$.
\end{theorem}
\begin{proof}
 The result follows by applying inverse cosine transform \eqref{Inversion_cosine_odd dim} and \eqref{Inversion_cosine_even dim} to equality \eqref{cos of radon}.
\end{proof}

\begin{remark}\label{R:comptonappl}\indent
\begin{enumerate}
\item
For any $\omega \in S^{n-1}$ and $s \in \mathbb{R}$, the Radon transform $Rf(\omega, s)$ of a function $f \in \mathcal{S}(\mathbb{R}^n)$ can be computed using formulas \eqref{Radon by cone_odd} and \eqref{Radon by cone_even}, if for any $(\omega, s)$ one has access to a cone vertex (= detector location) $u \in \mathbb{R}^n$ such that $u \cdot \omega=s$. For instance a line (curve) array of detectors should be sufficient. Thus, Theorem \ref{Radon in terms of cone} together with formula \eqref{inverse_radon} should provide inversion formulas for the cone transform that are applicable to Compton camera data. This idea  leads to a variety of new inversion formulas for Compton camera imaging, which will be derived and applied elsewhere.
\item We applied this approach to some 2D examples. Figures \ref{fig:reconstruction1} and \ref{fig:reconstruction2} show the reconstructions of some phantoms from their projections collected by four Compton cameras placed along the sides of a square. We simulate analytically the Compton projection data of the phantoms and then use formula \eqref{Radon by cone_even} to convert them to Radon projections. Finally, the filtered back-projection is applied to invert the Radon transform and obtain the reconstructions.
\end{enumerate}
\end{remark}

\begin{center}
\begin{figure}[H]
                \includegraphics[width=\textwidth]{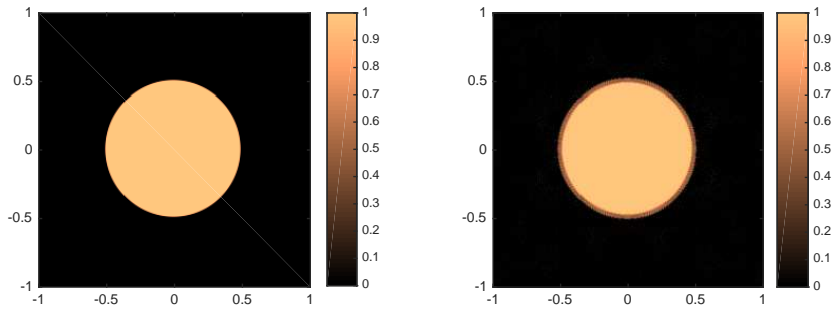}
                \caption{ Left: The phantom is the characteristic function of a circle having density 1 unit, radius 0.5 unit and centered at $(0, 0)$. Right: 256x256 image reconstructed from the simulated Compton data using 257 detectors per side and 200 counts for the angles $\beta$ and $\psi$ each (see Fig. \ref{fig:2Dcone}).}\label{fig:reconstruction1}

\end{figure}

\begin{figure}[H]
                \includegraphics[width=\textwidth]{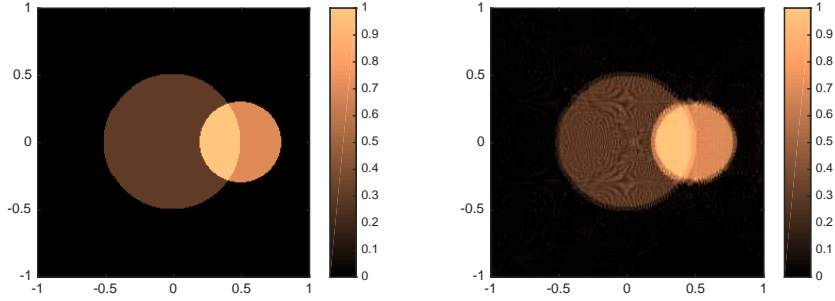}
               \caption{Left: The phantom is the sum of the characteristic functions of two intersecting circles having densities 0.3 and 0.7 units, radii 0.5 and 0.3 units, and centered at $(0,0)$ and $(0.5,0)$. Right: 256x256 image reconstructed from the simulated Compton data using 257 detectors per side and 200 counts for the angles $\beta$ and $\psi$ each (see Fig. \ref{fig:2Dcone}).}\label{fig:reconstruction2}
\end{figure}
\end{center}
\section{Relation of the Cone Transform with Spherical Harmonics}

Utilizing the relation of the cosine transform with spherical harmonics, we can relate the coefficients of the spherical harmonics expansion of the cone and Radon transforms.
\begin{lemma}
Let $g \in L^1(S^{n-1})$. Then,
\begin{equation}
\label{integration wrt beta}
 \int\limits_{S^{n-1}} \int\limits_0^\pi Cf(u,\beta,\psi)g(\beta) \sin\psi d\psi d\beta = \pi \int\limits_{S^{n-1}} Rf(\omega,\omega \cdot u)\textswab{C}g(\omega)d\omega.
\end{equation}
\end{lemma}
\begin{proof}
Multiplying both sides of \eqref{int_rel} with $g(\beta)$ and integrating with respect to $\beta$ over $S^{n-1}$, we have
\begin{align*}
 \int\limits_{S^{n-1}} \int\limits_0^\pi Cf(u,\beta,\psi)g(\beta) \sin\psi d\psi d\beta
 &=\frac{\pi}{|S^{n-1}|} \int\limits_{S^{n-1}} Rf(\omega,\omega \cdot u)\int\limits_{S^{n-1}}g(\beta)|\omega \cdot \beta| d\beta d\omega\\
 &=\pi \int\limits_{S^{n-1}} Rf(\omega,\omega \cdot u)\textswab{C}g(\omega)d\omega.
\end{align*}
\end{proof}
The spherical harmonics are known to be the eigenfunctions of the cosine transform. This follows from the Funk-Hecke Formula:
\begin{theorem}[\cite{Muller}]({Funk-Hecke Formula})
 \label{Funk-Hecke}
  Suppose $f(t)$ is continuous for $t \in [-1,1]$. Then, for every spherical harmonic $Y_m$ of degree $m$ and $\omega \in S^{n-1}$,
\begin{equation}
 \int\limits_{S^{n-1}}f(\omega \cdot \sigma)Y_m(\sigma)d\sigma=\lambda_mY_m(\omega),
\end{equation}
with
\begin{equation*}
 \lambda_m=|S^{n-2}|\int\limits_{-1}^1f(t)P_m(t)(1-t^2)^{(n-3)/2}dt,
\end{equation*}
where $P_m(t)$ are the Legendre polynomials, see \cite{Muller}.
\end{theorem}

\begin{corollary}
 For every spherical harmonic $Y_m$ of degree $m$, $m=0,1,2,...$, and every $\omega \in S^{n-1}$,
 \begin{equation}
 \label{Spherical Harmonics are evectors of cosine transform}
  \textswab{C}Y_m(\omega)=\lambda_mY_m(\omega)
 \end{equation}
where $\lambda_m$ is given as in Funk-Hecke Formula for $f(t)=|t|$.
\end{corollary}
Now, we can establish the following relation.
\begin{prop}
 For every spherical harmonic $Y_m$ of degree $m$,
 \begin{equation}\label{sph_har_rel}
  \int\limits_{S^{n-1}} \int\limits_0^\pi Cf(u,\beta,\psi)Y_m(\beta) \sin\psi d\psi d\beta =  \pi \lambda_m \int\limits_{S^{n-1}} Rf(\omega,\omega \cdot u)Y_m(\omega)d\omega.
 \end{equation}
In particular, for $m=0$, we obtain \eqref{BPR_rel}.
\end{prop}
\begin{proof}
Letting $g=Y_m$ in \eqref{integration wrt beta}, and using \eqref{Spherical Harmonics are evectors of cosine transform}, we get \eqref{sph_har_rel}. Then, the equation \eqref{BPR_rel} follows from direct calculation.
\end{proof}
\begin{remark}
As the relation \eqref{sph_har_rel} gives the spherical harmonics coefficients of the function $Rf(\omega, u\cdot \omega)$, one can recover it for all $u \in \mathbb{R}^n$ and $\omega \in S^{n-1}$. Then, any inversion formula for the Radon transform \eqref{inverse_radon} would reconstruct the function $f$.
This can be considered as an analog of Cormack's method \cite{Natt_old}.
\end{remark}

\section{Proofs of Some Auxiliary Statements}
\subsection{An Integral Relation for the Radon Transform}
\begin{lemma}
For any $f \in \mathcal{S}(\mathbb{R}^n)$, $u \in \mathbb{R}^n$, and $p \in \mathbb{R}$,
\begin{equation}
\label{Asgeirsson}
 \int\limits_{S^{n-1}}Rf(\omega,p+u \cdot \omega) d\omega=|S^{n-2}| \int\limits_{S^{n-1}} \int\limits_{|p|}^\infty f(u+r\omega)(r^2-p^2)^{(n-3)/2}rdrd\omega.
\end{equation}
\end{lemma}
\begin{proof}
Due to the shift invariance of the Radon transform, it suffices to prove the lemma for $u=0$ only. Let $F$ be the spherical mean-value of $f$, i.e.,
$$F(r)=\frac{1}{|S^{n-1}|}\int\limits_{S^{n-1}}f(r\omega) d\omega.$$

The rotational invariance of the Radon transform implies that it commutes with the spherical mean-value operator. Thus,
$$\hat{F}(p):= RF(\omega, p)= \frac{1}{|S^{n-1}|} \int\limits_{S^{n-1}} Rf(\xi,p)d\xi.$$

On the other hand, if $\{\omega,\omega_1^\bot,...,\omega_{n-1}^\bot\}$ is an orthonormal system in $\mathbb{R}^n$,
\begin{align*}
\label{}
\hat{F}(p)&=\int\limits_{-\infty}^\infty \cdots  \int\limits_{-\infty}^\infty F(p\omega+t_1\omega_1^\bot+\cdots+ t_{n-1}\omega_{n-1}^\bot)dt_1...dt_{n-1} \\
    &=\int\limits_{-\infty}^\infty \cdots  \int\limits_{-\infty}^\infty F\left(\sqrt{p^2+t_1^2+\cdots t_{n-1}^2}\right)dt_1...dt_{n-1}
\end{align*}
as $F$ is radial. Letting $x=t_1\omega_1^\bot+\cdots+ t_{n-1}\omega_{n-1}^\bot$, we have
\begin{align*}
\int\limits_{-\infty}^\infty \cdots  \int\limits_{-\infty}^\infty F\left(\sqrt{p^2+t_1^2+\cdots t_{n-1}^2}\right)dt_1...dt_{n-1}
&=\int\limits_{\mathbb{R}^{n-1}} F(\sqrt{p^2+|x|^2})dx\\
&=|S^{n-2}|\int\limits_0^\infty F(\sqrt{p^2+t^2})t^{n-2}dt.
\end{align*}

Finally, letting $r=\sqrt{p^2+t^2}$, we obtain
\begin{align*}
\int\limits_0^\infty F(\sqrt{p^2+t^2})t^{n-2}dt &= \int\limits_{|p|}^\infty F(r)(r^2-p^2)^{(n-3)/2} r dr\\
 &=\frac{1}{|S^{n-1}|} \int\limits_{S^{n-1}} \int\limits_{|p|}^\infty f(r\omega) (r^2-p^2)^{(n-3)/2} r dr d\omega.
\end{align*}

Hence, the result follows.
\end{proof}
\begin{corollary}[\cite{Natt_old}]\label{BPR}
Letting $p=0$ in \eqref{Asgeirsson}, we obtain
 \begin{align}
  R^\#Rf(u)&=\int\limits_{S^{n-1}}Rf(\omega,u \cdot \omega) d\omega \nonumber
  =|S^{n-2}| \int\limits_{S^{n-1}} \int\limits_{0}^\infty f(u+r\omega)r^{n-2}drd\omega\\
  &=|S^{n-2}| \int\limits_{\mathbb{R}^n} f(u+x)|x|^{-1}dx= |S^{n-2}| (|x|^{-1} \ast f)(u).
 \end{align}
\end{corollary}

\subsection{Proof of Proposition \ref{int_rel_vertical}}
\label{subsec:proof of prop}
We first prove the proposition for $n=2$. By definition of the 2-dimensional cone transform \eqref{2D_cone}, we have
\begin{align*}
\int\limits_0^\pi Cf(0, e_2,\psi)\sin \psi d\psi&= \int\limits_0^\pi \int\limits_0^\infty f(r\sin \psi, r\cos \psi)\sin\psi dr d\psi\\
&+\int\limits_0^\pi \int\limits_0^\infty f(-r\sin \psi, r\cos \psi)\sin\psi dr d\psi.
\end{align*}

Changing variables by letting $r \to -r$ and $\psi \to \pi-\psi$, respectively, we obtain
\begin{align*}
\int\limits_0^\pi \int\limits_0^\infty f(r\sin \psi, r\cos \psi)\sin\psi dr d\psi&=\int\limits_0^\pi \int\limits_{-\infty}^0 f(-r\sin \psi, -r\cos \psi)\sin\psi dr d\psi\\
&=\int\limits_0^\pi \int\limits_{-\infty}^0 f(-r\sin \phi, r\cos \phi)\sin\phi dr d\phi.
\end{align*}

Therefore,
\begin{align*}
\int\limits_0^\pi Cf(0, e_2,\psi)\sin \psi d\psi&=\int\limits_0^\pi \int\limits_{-\infty}^\infty f(-r\sin \psi, r\cos \psi)\sin\psi dr d\psi\\
&=\int\limits_0^\pi Rf(\omega(\psi),0)\sin\psi d\psi
\end{align*}
where $\omega(\psi):=(\cos \psi, \sin \psi)$. Now, the evenness property of the Radon transform implies that
\begin{align*}
\int\limits_0^\pi Rf(\omega(\psi),0)\sin\psi d\psi=\int\limits_0^\pi Rf(\omega(\psi+\pi),0)\sin\psi d\psi=-\int\limits_\pi^{2\pi} Rf(\omega(\phi),0)\sin\phi d\phi.
\end{align*}

Hence, we get
$$\int\limits_0^\pi Cf(0, e_2,\psi)\sin \psi d\psi=\frac{1}{2}\int\limits_0^{2\pi} Rf(\omega(\psi),0)|\sin\psi| d\psi=\frac{1}{2}\int\limits_{S^1} Rf(\omega,0) |\omega \cdot e_2| d\omega,$$
which is the equation \eqref{int rel vertical} for $n=2$.

In order to prove the proposition for $n \geq 3$, we need two auxiliary results.
\begin{lemma}\label{Integral Relation_CR}

For $\psi_0 \in (0, \pi/2)$, $\psi \in (0, \pi)$, and $n \geq 3$, we define
\begin{equation}\label{def of g}
g(\psi_0,\psi)=\frac{(\cos^2\psi_0-\cos^2\psi)^{(n-4)/2}}{(\sin\psi)^{n-3}}.
\end{equation}

Then, for any $f \in \mathcal{S}(\mathbb{R}^n)$,
\begin{equation}\label{Int_Rel_Lemma}
\int\limits_{\psi_0}^{\pi-\psi_0} Cf(0, e_n, \psi)g(\psi_0,\psi)d\psi= \frac{(\cos\psi_0)^{n-3}}{|S^{n-3}|}\int\limits_{S^{n-2}}Rf((\cos\psi_0)\omega,\sin \psi_0),0)d\omega.
\end{equation}
\end{lemma}

\begin{proof} The idea of the proof is to exhaust the exterior volume of two opposite cones having a common vertex in two ways. The first is by taking a family of vertical cones whose vertices are at the origin and opening angles vary from $\psi_0$ to $\pi-\psi_0$. The second is to consider a family of hyperplanes passing through origin and are tangent to the vertical cone having vertex at the origin and opening angle $\psi_0$ (See Fig. \ref{fig:geometry of proof}).

\begin{figure}[H]
\begin{center}
\includegraphics[width=4in,height=2.7in]{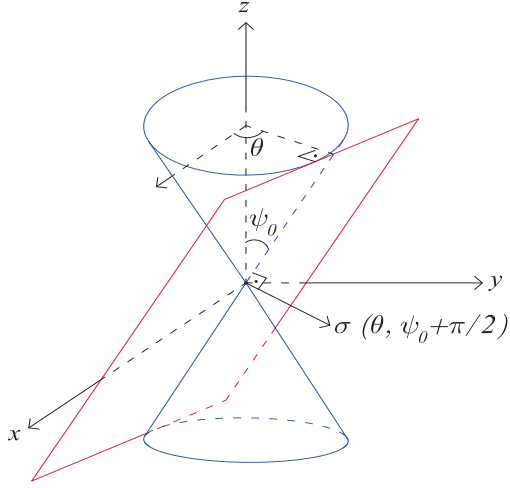}
\caption{Geometry of Lemma \ref{Integral Relation_CR}.}
\label{fig:geometry of proof}
\end{center}
\end{figure}

Let the functions $f$ and $g$ be given as in the lemma. We can split the integral on the left hand side of equation \eqref{Int_Rel_Lemma} into two parts to get
\begin{align} \label{int_split}
\int\limits_{\psi_0}^{\pi-\psi_0} Cf(0, e_n, \psi)g(\psi_0,\psi)d\psi&=\int\limits_{\psi_0}^{\pi/2} Cf(0, e_n, \psi)g(\psi_0,\psi)d\psi \nonumber\\
&+\int\limits_{\pi/2}^{\pi-\psi_0} Cf(0, e_n, \psi)g(\psi_0,\psi)d\psi.
\end{align}

By the definition of the vertical cone transform \eqref{nD_vertical cone}, for the first term on the right hand side, then
\begin{align*}
\int\limits_{\psi_0}^{\pi/2}&Cf(0,e_n,\psi)g(\psi_0,\psi)d\psi \\
&= \int\limits_{\psi_0}^{\pi/2} \int\limits_0^{\infty} \int\limits_{S^{n-2}}f(\rho (\sin\psi) \omega, \rho \cos \psi)(\rho \sin\psi)^{n-2}g(\psi_0,\psi)d\omega d\rho d\psi.
\end{align*}

If we make a change of variables in the integral with respect to $\rho$ by letting $z= \rho \cos \psi$, we have
\begin{align*}
\int\limits_{\psi_0}^{\pi/2}&Cf(0,e_n,\psi)g(\psi_0,\psi)d\psi \\
&= \int\limits_{\psi_0}^{\pi/2} \int\limits_0^{\infty} \int\limits_{S^{n-2}}f(z\tan\psi \omega,z)(z\tan\psi)^{n-2}g(\psi_0,\psi)d\omega \frac{dz}{\cos\psi} d\psi.
\end{align*}

Now, if we let $r=z\tan \psi$, then $dr=z \sec^2 \psi d\psi$, and since
\begin{align*}
\cos^2\psi_0-\cos^2\psi=\frac{\sec^2\psi-\sec^2\psi_0}{\sec^2\psi_0 \sec^2\psi}=\frac{\tan^2\psi-\tan^2\psi_0}{\sec^2\psi_0 \sec^2\psi} = \frac{r^2-z^2\tan^2\psi_0}{z^2\sec^2\psi_0 \sec^2\psi},
\end{align*}
we have $g(\psi_0,\psi(r,z))=\frac{(r^2-z^2\tan^2\psi_0)^{(n-4)/2}}{(r \sec\psi_0)^{n-4}}.$

Thus, 
\begin{align*}
 \int\limits_{\psi_0}^{\pi/2}&Cf(0,e_n,\psi)g(\psi_0,\psi)d\psi \\
 &= (\cos\psi_0)^{n-4}\int\limits_0^\infty \int\limits_{S^{n-2}}\int\limits_{z\tan\psi_0}^{\infty} f_z(r\omega)(r^2-z^2\tan^2\psi_0)^{(n-4)/2}r dr d\omega dz.
\end{align*}

Then, using the identity \eqref{Asgeirsson}, we obtain the following relation between the cone transform of $f$ and $(n-1)$-dimensional Radon transform of $f_z$.
\begin{align*}
\int\limits_{\psi_0}^{\pi/2}Cf(0,e_n,\psi)&g(\psi_0,\psi)d\psi=\frac{(\cos\psi_0)^{n-4}}{|S^{n-3}|} \int\limits_0^\infty \int\limits_{S^{n-2}} Rf_z(\omega,-z\tan\psi_0) d\omega dz\\
&=\frac{(\cos\psi_0)^{n-4}}{|S^{n-3}|} \int\limits_0^\infty \int\limits_{S^{n-2}} \int\limits_{\mathbb{R}^{n-1}} f_z(\bar{x}) \delta(\bar{x}\cdot \omega+z\tan\psi_0)d\bar{x} d\omega dz.
\end{align*}

Now, since $\delta(\lambda (u-a))=\lambda^{-1}\delta(u-a)$, we get
\begin{align}\label{upperhalf}
\int\limits_{\psi_0}^{\pi/2}&Cf(0,e_n,\psi)g(\psi_0,\psi)d\psi \nonumber \\
&=\frac{(\cos\psi_0)^{n-3}}{|S^{n-3}|} \int\limits_{S^{n-2}}\int\limits_0^\infty \int\limits_{\mathbb{R}^{n-1}} f(\bar{x},z) \delta(\bar{x}\cdot (\cos\psi_0)\omega+z\sin\psi_0)d\bar{x}dz d\omega.
\end{align}

For the second term of the right hand side of \eqref{int_split}, we change the variable $\psi$ by $\pi-\psi$ to get
\begin{align*}
\int\limits_{\pi/2}^{\pi-\psi_0}&Cf(0,e_n,\psi)g(\psi_0,\psi)d\psi =\int\limits_{\psi_0}^{\pi/2}Cf(0,e_n,\pi-\psi)g(\psi_0,\pi-\psi)d\psi\\
&= \int\limits_{\psi_0}^{\pi/2} \int\limits_0^{\infty} \int\limits_{S^{n-2}}f(\rho (\sin\psi) \omega, -\rho \cos \psi)(\rho \sin\psi)^{n-2}g(\psi_0,\psi)d\omega d\rho d\psi.
\end{align*}

Again we change variables first by letting $z=\rho \cos \psi$ and then $r=z\tan \psi$ to obtain
\begin{align*}
\int\limits_{\pi/2}^{\pi-\psi_0}&Cf(0,e_n,\psi)g(\psi_0,\psi)d\psi\\
&= (\cos\psi_0)^{n-4}\int\limits_0^\infty \int\limits_{S^{n-2}}\int\limits_{z\tan\psi_0}^{\infty} f_{-z}(r\omega)(r^2-z^2\tan^2\psi_0)^{(n-4)/2}r dr d\omega dz\\
&=\frac{(\cos\psi_0)^{n-4}}{|S^{n-3}|} \int\limits_0^\infty \int\limits_{S^{n-2}} Rf_{-z}(\omega, z\tan\psi_0) d\omega dz,
\end{align*}
where the last equality follows from the identity \eqref{Asgeirsson}.  Again, by the definition of the Radon transform, and $\delta(\lambda (u-a))=\lambda^{-1}\delta(u-a)$, we get
\begin{align*}
\int\limits_0^\infty \int\limits_{S^{n-2}} &Rf_{-z}(\omega, z\tan\psi_0) d\omega dz\\
&=\cos\psi_0\int\limits_{S^{n-2}}\int\limits_0^\infty \int\limits_{\mathbb{R}^{n-1}} f(\bar{x},-z) \delta(\bar{x}\cdot (\cos\psi_0)\omega-z\sin\psi_0)d\bar{x}dz d\omega\\
&=\cos\psi_0\int\limits_{S^{n-2}}\int\limits_{-\infty}^0 \int\limits_{\mathbb{R}^{n-1}} f(\bar{x},z) \delta(\bar{x}\cdot (\cos\psi_0)\omega+z\sin\psi_0)d\bar{x}dz d\omega.
\end{align*}

Thus,
\begin{align}\label{lowerhalf}
\int\limits_{\pi/2}^{\pi-\psi_0}&Cf(0,e_n,\psi)g(\psi_0,\psi)d\psi \nonumber\\
&=\frac{(\cos\psi_0)^{n-3}}{|S^{n-3}|}\int\limits_{S^{n-2}}\int\limits_{-\infty}^0 \int\limits_{\mathbb{R}^{n-1}} f(\bar{x},z) \delta(\bar{x}\cdot (\cos\psi_0)\omega+z\sin\psi_0)d\bar{x}dz d\omega.
\end{align}

Now, using \eqref{upperhalf} and \eqref{lowerhalf} for the first and second terms in the equation \eqref{int_split}, we obtain
\begin{align*}
\int\limits_{\psi_0}^{\pi-\psi_0} &Cf(0, e_n, \psi)g(\psi_0,\psi)d\psi\\
&=\frac{(\cos\psi_0)^{n-3}}{|S^{n-3}|} \int\limits_{S^{n-2}}\int\limits_{\mathbb{R}^n} f(x) \delta(x\cdot ((\cos\psi_0)\omega,\sin\psi_0))dx d\omega.
\end{align*}

Finally, observing that
 $$\int\limits_{\mathbb{R}^n} f(x) \delta(x\cdot ((\cos\psi_0)\omega,\sin\psi_0))dx=Rf(((\cos\psi_0)\omega,\sin\psi_0),0),$$
we have
 \begin{align*}
\int\limits_{\psi_0}^{\pi-\psi_0} Cf(0, e_n, \psi)g(\psi_0,\psi)d\psi= \frac{(\cos\psi_0)^{n-3}}{|S^{n-3}|}\int\limits_{S^{n-2}}Rf((\cos\psi_0)\omega,\sin \psi_0),0)d\omega.
\end{align*}

Hence, we get the result.
\end{proof}

\begin{lemma}
\label{Derivative of Integral}
Assume that $n \geq 3$. Let $g(\psi_0,\psi)$ be given as in \eqref{def of g} and define
$$h(\psi_0,\psi)=\frac{(\cos^2\psi_0-\cos^2\psi)^{(n-2)/2}}{(\sin\psi)^{n-3}}.$$
Then,
\begin{align*}
\frac{d}{d\psi_0}\int\limits_{\psi_0}^{\pi-\psi_0}&Cf(0,e_n,\psi)h(\psi_0,\psi)d\psi\\
&= (2-n)\cos \psi_0 \sin\psi_0  \int\limits_{\psi_0}^{\pi-\psi_0}Cf(0,e_n,\psi)g(\psi_0,\psi)d\psi.
\end{align*}
\end{lemma}
\begin{proof}
 As $\frac{\partial h}{\partial \psi_0}(\psi_0,\psi)=(2-n)\cos \psi_0 \sin\psi_0 g(\psi_0,\psi)$,  utilizing Leibniz integral rule and noticing that $h(\psi_0,\pi-\psi_0)=h(\psi_0,\psi_0)=0$  gives the result.
\end{proof}

\emph{Proof of Proposition \ref{int_rel_vertical}, $n \geq 3$}. By Lemmas \ref{Derivative of Integral} and \ref{Integral Relation_CR}, we have
\begin{align*}
 \frac{d}{d\psi_0}\int\limits_{\psi_0}^{\pi-\psi_0}&Cf(0,e_n,\psi)h(\psi_0,\psi)d\psi\\
&= (2-n)\cos \psi_0 \sin\psi_0  \int\limits_{\psi_0}^{\pi-\psi_0}Cf(0,e_n,\psi)g(\psi_0,\psi)d\psi\\
&=\frac{(2-n)sin\psi_0(\cos\psi_0)^{n-2}}{|S^{n-3}|}\int\limits_{S^{n-2}}Rf(((\cos\psi_0)\omega,\sin \psi_0),0)d\omega.
\end{align*}

Integrating both sides with respect to $\psi_0$ from $0$ to $\pi/2$, we obtain
\begin{align}\label{zero to half pi}
\begin{split}
\int\limits_0^{\pi}&Cf(0,e_n,\psi)\sin\psi d\psi\\
&=\frac{n-2}{|S^{n-3}|}\int\limits_0^{\pi/2}\int\limits_{S^{n-2}}Rf(((\cos\psi_0)\omega,\sin \psi_0),0)d\omega \sin\psi_0(\cos\psi_0)^{n-2} d\psi_0\\
&=\frac{n-2}{|S^{n-3}|}\int\limits_0^{\pi/2}\int\limits_{S^{n-2}}Rf(((\sin \phi) \omega,\cos \phi),0) \cos\phi(\sin\phi)^{n-2} d\omega d\phi,
\end{split}
\end{align}
where we changed the variable by letting $\phi=\frac{\pi}{2}-\psi_0$. On the other hand, letting $\phi=\psi_0+\frac{\pi}{2}$, we have
\begin{align*}
 \int\limits_0^{\pi}Cf&(0,e_n,\psi)\sin\psi d\psi\\
&=\frac{n-2}{|S^{n-3}|}\int\limits_0^{\pi/2}\int\limits_{S^{n-2}}Rf(((\cos\psi_0)\omega,\sin \psi_0),0)d\omega \sin\psi_0(\cos\psi_0)^{n-2} d\psi_0\\
&=\frac{n-2}{|S^{n-3}|}\int\limits_{\pi/2}^\pi \int\limits_{S^{n-2}}Rf(((\sin \phi) \omega,-\cos \phi),0) (-\cos\phi)(\sin\phi)^{n-2} d\omega d\phi.
\end{align*}

Now, due to evenness of Radon transform, we have
\begin{align*}
 Rf(((\sin \phi) \omega,&-\cos \phi),0)=Rf(((-\sin \phi) (-\omega),-\cos \phi),0)\\
 &=Rf(-((\sin \phi) (-\omega),\cos \phi),0)=Rf(((\sin \phi)(-\omega),\cos \phi),0).
\end{align*}
Since the Lebesgue measure is rotation invariant, we obtain
\begin{align}\label{half pi to pi}
\begin{split}
 \int\limits_0^{\pi}Cf&(0,e_n,\psi)\sin\psi d\psi\\
&=\frac{n-2}{|S^{n-3}|}\int\limits_{\pi/2}^\pi \int\limits_{S^{n-2}}Rf(((\sin \phi )\omega,\cos \phi),0) (-\cos\phi)(\sin\phi)^{n-2} d\omega d\phi.
\end{split}
\end{align}

Summing \eqref{zero to half pi} and \eqref{half pi to pi}, we conclude that
\begin{align*}
 \int\limits_0^{\pi}Cf&(0,e_n,\psi)\sin\psi d\psi\\
&=\frac{n-2}{2|S^{n-3}|}\int\limits_0^\pi \int\limits_{S^{n-2}}Rf(((\sin \phi)\omega,\cos \phi),0) |\cos\phi|(\sin\phi)^{n-2} d\omega d\phi\\
&=\frac{n-2}{2|S^{n-3}|}\int\limits_{S^{n-1}}Rf(\sigma,0) |\sigma\cdot e_n|d\sigma.
\end{align*}

Finally, application of the formula \eqref{area of nsphere} and $\Gamma(z+1)=z\Gamma(z)$ gives the result.\qed

\section{Conclusions and Remarks}\label{sec:Remarks}

In this paper, various relations between the general (overdetermined) cone transform and Radon and cosine transforms and spherical harmonic expansions are explored.
Several inversion formulas for the cone transform are obtained, some of which of filtered backprojection nature. Examples of reconstructions from synthetic Compton camera data are provided.

Some additional remarks: 
\begin{itemize}
\item In order not to distract from the main point, the source intensity distribution function $f$ is assumed to be of the Schwartz class, $\mathcal{S}(\mathbb{R}^n)$. In fact, the cone transform of $f$ is well-defined even when we assume integrability of $f$ on each cone. The formulas obtained here can be extended by continuity to much larger function spaces. For instance, for the inversion formula \eqref{Inversion} to hold, it is sufficient that the function $Cf$ is $(n-1)$-times differentiable with respect to $u$, and to this end it suffices to assume the function $f$ be $(n-1)$-times differentiable. As a condition of decaying, assuming that $f(x)=\mathcal{O}(|x|^{-N})$ for some $N >n$, is sufficient.
\item Although we do not explicitly present the adjoint of the cone transform, both Theorem \ref{Inversion Theorem1} and Theorem \ref{Inversion Thm} provide filtered back projection type inversion formulas for the cone transform as, in both cases, we recover the function at a point $u$ using a weighted averaging of its cone transform over cones having vertex at $u$.
\item Let us address the comparison of inversion formulas of Theorems \ref{Inversion Theorem1} and \ref{Inversion Thm}. Both of them involve integrating the data with respect to $\psi$ and $\beta$ and filtration by the same Riesz potential. The difference is that in Theorem \ref{Inversion Theorem1} the measure of integration is $\mu(\beta)d\psi d\beta$ with arbitrary function $\mu$ of mass $1$ (e.g., a $\delta$-function), while the formula of Theorem \ref{Inversion Thm} holds only for the measure $\sin(\psi)d\psi d\beta$.
\end{itemize}

\section{Acknowledgements}
The author is grateful to P. Kuchment who provided insight and expertise that greatly assisted the research in this paper. The author is also thankful to Y. Hristova, L. Kunyansky, S. Moon and B. Rubin for helpful comments, discussions, and references. Finally, the author is grateful to the referees for careful review of the paper and for the comments, corrections and suggestions that lead to significant improvements of the paper. This work was partially supported by the NSF DMS grant 1211463.

\end{document}